\documentclass[12pt]{amsart}
\usepackage[paper=a4paper,left=3.0cm,right=3.0cm,top=3.7cm,bottom=3.7cm]{geometry}
\usepackage{float, graphicx, amssymb, dsfont}
\usepackage[small,nohug,heads=vee]{diagrams}

\newtheorem{thm}{Theorem}[section]
\newtheorem{cor}[thm]{Corollary}
\newtheorem{lem}[thm]{Lemma}
\newtheorem{prop}[thm]{Proposition}

\theoremstyle{definition}
\newtheorem{defn}{Definition}[section]
\newtheorem{exam}[defn]{Example}

\theoremstyle{remark}
\newtheorem{rem}[defn]{Remark}
\numberwithin{equation}{section}
\newfloat{Figure}{h}{}[section]

\newcommand{\abs}[1]{\left\vert#1\right\vert}
\newcommand{\set}[1]{\left\{#1\right\}}
\newcommand{\A}{\mathcal{A}}    
\newcommand{\B}{\mathcal{B}}    
\newcommand{\D}{\mathcal{D}}
\newcommand{\F}{\mathcal{F}}    
\newcommand{\V}{\mathcal{V}}    
\newcommand{\E}{\mathcal{E}}    
\newcommand{\setS}{\mathcal{S}}    

\newcommand{\G}{\text{\textsf{G}}}
\newcommand{\X}{\text{\textsf{X}}}

\newcommand{\ifff}{if and only if}

\newcommand{\wee}{with equal entropy}

\newcommand{\To}{\Rightarrow}

\newcommand{\cto}{constant-to-one}
\newcommand{\fto}{finite-to-one}

\newcommand{\rr}{right resolving}

\newcommand{\bir}{bi-resolving}
\newcommand{\rc}{right closing}
\newcommand{\lc}{left closing}
\newcommand{\bic}{bi-closing}

\newcommand{\pXtoY}{\phi : X \to Y}

\newcommand{\pStoY}{\phi : \Sigma \to Y}
\newcommand{\preimage}[2] {{#1}^{-1}(#2)}

\newcommand{\preimageAbs}[2] {|{#1}^{-1}(#2)|}

\newcommand{\SFT}{shift of finite type}
\newcommand{\SFTs}{shifts of finite type}
\newcommand{\rSFT}{irreducible shift of finite type}
\newcommand{\rSFTs}{irreducible shifts of finite type}
\newcommand{\Sofic}{sofic shift}
\newcommand{\rSofic}{irreducible sofic shift}

\newcommand{\TFAE}{the following are equivalent}

\newcommand{\setN}{\mathbb{N}}

\newcommand{\setZ}{\mathbb{Z}}

\begin{document}

\title[Open maps between shift spaces]{Open maps between shift spaces}
\author{Uijin Jung}
\address{Department of Mathematical Sciences, Korea Advanced Institute of Science and Technology, Daejeon 305-701, Korea}
\email{uijin@kaist.ac.kr}

\thanks{2000 \textit{Mathematics Subject Classification.} Primary 37B10; Secondary 37B15, 37B40, 54H20.}
\thanks{This work was supported by the second stage of the Brain Korea 21 Project, The Development Project of Human Resources in Mathematics, KAIST in 2008. }
\keywords{sofic shift, shift of finite type, open, \cto, \bic}


\begin{abstract}
    Given a code from a shift space to an \rSofic, any two of the following three conditions -- open, \cto, (right or left) closing -- imply the third. If the range is not sofic, then the same result holds when \bic ness replaces closingness. Properties of open mappings between shift spaces are investigated in detail. In particular, we show that a closing open (or \cto) extension preserves the structure of a sofic shift.
\end{abstract}
\maketitle

\section{Introduction and Preliminaries}

A map between topological spaces is open if images of open sets are open. In symbolic dynamics, the openness of a sliding block code was first considered in \cite{Hed} and has been proved useful, especially in cellular automata \cite{Kur}. Hedlund showed that an endomorphism of a full shift is open if and only if it is \cto, i.e., the cardinality of the preimage of every point is same. In \cite{Nas}, Nasu extended this result to the \fto\ codes between \rSFTs\ and showed further that in that category, these conditions are also equivalent to the \bic ness of a code. In the sofic cases, the situation is different and any single condition does not imply the others in general. So far, the structure of open codes between \Sofic s are not well understood. In this paper, we investigate the properties of such codes and extend the result of Nasu to the class of sofic shifts. The main result is stated as follows.

\begin{thm} \label{thm:Main1}
    Let $\phi$ be a code from a shift space $X$ to an \rSofic\ $Y$. Then any two of the following conditions imply the third:
        \begin{enumerate}
            \item $\phi$ is open.
            \item $\phi$ is \cto.
            \item $\phi$ is \rc\ (or \lc).
        \end{enumerate}
    In the case where any two (hence all) of the above conditions hold,  $X$ is a nonwandering \Sofic.
\end{thm}

When $Y$ is not necessarily sofic, the same result still holds if we replace (3) with the condition that $\phi$ is \bic\ (see Theorem \ref{thm:Main2}).
It will also be shown that when the domain in Theorem \ref{thm:Main1} is of finite type and $\phi$ is \fto, then the code is open \ifff\ it is \cto\ (see Theorem \ref{thm:Y_Sofic}).

In the early of 1990's, Blanchard and Hansel proved that an irreducible sofic \cto\ extension of an \rSFT\ is also of finite type \cite{BH}. If either soficity or irreducibility is not assumed, then the statement is false \cite{Fie}. It turns out, however, that we need not assume the extension to be sofic and irreducible if the extension map is closing, i.e., \rc\ or \lc. We will prove that a \cto\ closing extension of an \rSFT\ is also of finite type (see Corollary \ref{cor:rcopen_ext}). This result can be viewed as a generalization of \cite{BH}, since the proof in \cite{BH} is essentially showing that the extension map is closing. We give an analogue of the above result for \rc\ open extensions. Indeed, assumed to be closing, an open code and a \cto\ code have similar properties. Some other interesting properties of open codes are provided.

We assume that the reader is familiar with elementary symbolic dynamics. For an introduction, see \cite{Kit} or \cite{LM}. For a general theory of topological dynamics, see \cite{Kur} or \cite{Wal}.

Let $\A$ be a finite set with discrete topology and consider the set $\A^\setZ$ with the product topology. It is compact and metrizable, and a typical element is $x = (x_i)_{i \in \setZ}$, with $x_i \in \A$. Together with a homeomorphism $\sigma$, called a \textit{shift map}, defined by $\sigma(x)_i = x_{i+1}$, it becomes a topological dynamical system called a \textit{full shift}. A \textit{shift space} or \textit{subshift} is a $\sigma$-invariant closed subset $X$ of a full shift, together with a restriction of $\sigma$ to $X$. We may also refer to $X$ as a shift space.

If $X$ is a shift space, then denote by $\B_n(X)$ the set of all words of length $n$ appearing in the points of $X$ and $\B(X) = \bigcup_n \B_n(X)$. A shift space $X$ is called \textit{nonwandering} if for all $u \in \B(X)$, we can find a word $w$ such that $uwu \in \B(X)$. It is called \textit{irreducible} if for all $u, v \in \B(X)$, there is a word $w$ with $uwv \in \B(X)$. So $X$ is irreducible \ifff\ there is a point in $X$ in which every word in $\B(X)$ occurs infinitely often to the left and to the right. We call such a point \textit{doubly transitive}. Indeed, if $X$ is irreducible, then such points are dense in $X$ and $(X,\sigma)$ is topologically transitive.

For $u \in \B(X)$, let ${}_l[u]$ denote the open and closed set $\{x \in X : x_{[l,l+|u|-1]}=u\}$, which we call a \emph{cylinder}. If $l \geq 0$ and $|u|=2l+1$, then ${}_{-l}[u]$ is called a \emph{central} $2l+1$ \emph{cylinder}. When $u \in \B_1(X)$ and $l=0$, we usually discard the subscript 0.

A code $\pXtoY$ is a continuous $\sigma$-commuting map between shift spaces. Any code can be recoded to a 1-block code. A code $\phi$ is called a \textit{factor} (\textit{conjugacy}, resp.) if it is onto (bijective, resp.) If $\phi$ is a conjugacy, then $X$ and $Y$ are called \textit{conjugate}.

Let $d$ be the metric on $X$ by letting $d(x,y) = 2^{-k}$, where $k$ is the maximal number with $x_{[-k,k]} = y_{[-k,k]}$ if $x \neq y$, and 0 otherwise. Two points $x$ and $\bar x$ in $X$ are \textit{left asymptotic} if $d(\sigma^{-n}(x),\sigma^{-n}(\bar x)) \to 0$ as $n \to \infty$. A code $\phi$ is called \textit{right closing} if it never collapses two distinct left asymptotic points. A 1-block code $\phi$ is called \textit{right resolving} if whenever $ab$, $ac$ are in $\B_2(X)$ and $\phi(ab) = \phi(ac)$, then we have $b=c$. A right closing code can be recoded to a right resolving code. Right asymptotic points, and left closing (resolving, resp.) properties are defined similarly. We call $\phi$ \textit{closing} if it is left or right closing. If $\phi$ is both left and right closing (resolving, resp.), then it is called \textit{\bic} (\textit{\bir,} resp.) We call $\phi$ \textit{\fto} if $\preimage{\phi}{y}$ is a finite set for all $y \in Y$. Closing codes are \fto. A \fto\ code is called \textit{\cto} if $\preimageAbs{\phi}{y}$ is independent of $y$. Note that a \cto\ code is a factor code.

A subshift $X$ is called a \textit{shift of finite type} if there is a finite set $\F$ of words such that $X$ consists of all points in which there's no occurrence of words in $\F$.
If $A$ is a nonnegative integral matrix, then a directed graph $\G_A$ (with the vertex set $\V$ and the edge set $\E$) and a shift space $\X_A$ whose elements are the set of all bi-infinite trips in $\G_A \subset {\E}^{\setZ}$ are naturally associated to $A$. We call $\X_A$ an \textit{edge shift}. It is well known that any shift of finite type is conjugate to an edge shift.
On the vertex set $\V$ of a graph $\G_A$, define an equivalence relation by setting $v \sim w$ if there is a path from $v$ to $w$ and vice versa. The vertices in each equivalence class, together with all edges whose endpoints are in this equivalence class, form a natural subgraph of $\G_A$, called an (irreducible) component of $\G_A$. For each component of $\G_A$, a subshift of $\X_A$ is naturally defined. If it is not empty, we call this subshift an \textit{(irreducible) component} of $\X_A$. It is irreducible and there is no subshift of $\X_A$ which is irreducible and contains a component of $\X_A$ properly. A component $X_0$ of $\X_A$ is called a \textit{sink component} if every point left asymptotic to $X_0$ is also in $X_0$. A \textit{source component} is defined similarly. All concepts defined on edge shifts naturally extend to the \SFTs\ by conjugacies. In general, an irreducible subshift $X_0$ of $X$ is an \textit{(irreducible) component} of $X$ if there is no irreducible subshift of $X$ which properly contains $X_0$.

A shift space is called \textit{sofic} if it is a factor of a \SFT. If $X$ is an \rSofic, then there exist an \rSFT\ $\X_A$ and a right resolving factor code $\pi : \X_A \to X$ with the condition that any right closing factor code from an \rSFT\ to $X$ factors through $\pi$. This pair is unique up to conjugacy. We call $(\X_A,\pi)$ the \textit{canonical cover} of $X$.

Let $h(X)$ denote the topological entropy of a topological dynamical system $(X,\sigma)$. By the Perron-Frobenius theory, if $X$ is irreducible sofic and $Z$ is a proper subshift of $X$, then $h(Z) < h(X)$. For a shift space $X$, let $\Omega(X)$ be a (unique) maximal nonwandering subshift of $X$. If $X$ is of finite type, then $\Omega(X)$ is equal to the (disjoint) union of its components. By the variational principle in ergodic theory, if $Z$ is a subshift of $X$ and $\Omega(X) = \Omega(Z)$, then $h(X) = h(Z) = h(\Omega(X))$. Finite-to-one codes preserve entropy in the sense that $h(X) = h(\phi(X))$. A component $X_0$ of a shift space $X$ is called \emph{maximal} if $h(X_0) = h(X)$.

In this paper, our focus is mainly on open codes between shift spaces. A code $\pXtoY$ is called an open code if images of open sets are open. It is easy to see the following equivalent condition of openness of a code.

\begin{lem}\label{obs_open}
    A code $\pXtoY$ between shift spaces is open \ifff\ for each $l \in \setN$, there is $k \in \setN$ such that whenever $x \in X, y \in Y$ and $\phi(x)_{[-k,k]} = y_{[-k,k]}$, we can find $\bar{x} \in X$ with $\bar{x}_{[-l,l]} = x_{[-l,l]}$ and $\phi(x) = y$.
\end{lem}

So if $\phi$ is open, then for any $l \in \setN$ we can find $k \in \setN$ such that the image of a central $(2l+1)$ cylinder in $X$ consists of central $(2k+1)$ cylinders in $Y$. In the following we will often use these observations.

\vspace{1cm}

\section{Properties of open codes}\label{sec:openness}

In this section, we give some properties of open codes which are used in subsequent sections. We show that the image of a \SFT\ by an open code is also of finite type. Openness of a code is preserved by the fiber product. If combined with \cto\ property, an open code separates fibers and becomes a local homeomorphism.

\begin{lem}\label{lem:open_then_onto}
    Let $X$ and $Y$ be shift spaces with $Y$ irreducible. If $\pXtoY$ is an open code, then it is onto.
\end{lem}
\begin{proof}
    Note that $\phi(X)$ is a nonempty open and closed $\sigma$-invariant subset of $Y$. Since $(Y,\sigma)$ is topologically transitive, any nonempty open $\sigma$-invariant subset of $Y$ is dense in $Y$ \cite{Wal} and we get $\phi(X) = \overline {\phi(X)} = Y$. Thus $\phi$ is onto.
\end{proof}

\begin{lem} \cite{CP} \label{lem:cto_sft_then_sft}
    Let $X$ be an \rSFT\ and $\phi$ a \cto\ code from $X$ to a shift space $Y$. Then $Y$ is of finite type.
\end{lem}

This result on \cto\ codes is useful. Indeed, the code in the lemma is open (see Theorem \ref{thm:Nasu}).

\begin{prop}\label{prop:open_sft_then_sft}
    Let $X$ be a \SFT\ and $\phi$ an open factor code from $X$ to a shift space $Y$. Then $Y$ is of finite type.
\end{prop}
\begin{proof}
    By recoding, we can assume that $X$ is an edge shift and $\phi$ is 1-block. Choose $l \geq 0$ so that for each $a \in \B_1(X)$, $\phi([a])$ consists of central $2l+1$ cylinders.

    Let $uw$ and $wv$ be in $\B(Y)$ with $\abs{w} = 2l+1$. Then there are points $y^{(1)}$ and $y^{(2)}$ such that $y^{(1)}_{[-l-|u|,l]}=uw$ and $y^{(2)}_{[-l,l+|v|]}=wv$. Take $x \in \preimage{\phi}{y^{(1)}}$ and consider $\phi([x_0])$. It contains $y^{(1)}$ and $y^{(1)}_{[-l,l]} =y^{(2)}_{[-l,l]} = w$. Hence $y^{(2)} \in \phi([x_0])$ and there is a point $z \in [x_0]$ with $\phi(z) = y^{(2)}$. Now define ${\bar x}$ by ${\bar x}_i = x_i$ for $i \leq 0$ and $ {\bar x}_i = {z}_i$ for $i \geq 0$. Since $\phi(\bar x)_{[-l-|u|,l+|v|]} = uwv$, we have $uwv \in \B(Y)$. Thus $Y$ is a ($2l+1$)-step shift of finite type.
\end{proof}

Let $X$, $Y$, and $Z$ be shift spaces and $\phi_1 : X \to Z$, $\phi_2 : Y \to Z$ the codes. Then \emph{the fiber product} of $(\phi_1,\phi_2)$ is the triple $(\Sigma,\psi_1,\psi_2)$ where
    $$\Sigma = \set{ (x,y) \in X \times Y : \phi_1(x) = \phi_2(y) }$$
and $\psi_1 : \Sigma \to X$ is defined by $\psi_1(x,y) = x$; similarly for $\psi_2$.

Fiber product construction is useful since it lifts over or pulls down the properties of given codes to the opposite codes. More specifically, if $\phi_1$ is one-to-one, onto, \rc, \rr, or finite-to-one, then $\psi_2$ also satisfies the corresponding condition. In the case $\phi_2$ is assumed to be onto, if $\psi_2$ has any of the properties mentioned above, then so does $\phi_1$ \cite{LM}. It can be easily shown that \cto\ property is also preserved in this way.
Now we show that openness is also preserved by the fiber product.

\begin{lem} \label{lem:open_fiberproduct}
    Let $(\Sigma,\psi_1,\psi_2)$ be the fiber product of $\phi_1 : X \to Z$ and $\phi_2 : Y \to Z$.
    \begin{enumerate}
        \item
            If $\phi_1$ is open, then so is $\psi_2$.
        \item
            Let $\phi_2$ be onto. If $\psi_2$ is open, then so is $\phi_1$.
    \end{enumerate}
\end{lem}
\begin{proof}
    (1) Suppose $\phi_1$ is open. Let $C$ be the central $2l+1$ cylinder in $\Sigma$, containing a point $(x,y)$ with $x \in X$ and $y \in Y$. To show $\psi_2$ is open, it suffices to find an open neighborhood of $y$, contained in $\psi_2(C)$. Let $m$ be a coding length of $\phi_2$. Choose $n \ge l$ so that $\phi_1({}_{-l}[x_{-l} \cdots x_l])$ is a disjoint union of central $2n+1$ cylinders. Now, if $\bar y \in Y$ satisfies ${\bar y}_{[-n-m,n+m]} = y_{[-n-m,n+m]}$, then $\phi_2(\bar y)_{[-n,n]} = \phi_2(y)_{[-n,n]}$. Since $\phi_1({}_{-l}[x_{-l} \cdots x_l])$ contains $\phi_2(y)$ and the central $2n+1$ blocks of $y$ and $\bar y$ are equal, it follows that $\phi_2(\bar y) \in \phi_1({}_{-l}[x_{-l} \cdots x_l])$. Hence there is an $\bar x \in {}_{-l}[x_{-l} \cdots x_l]$ such that $\phi_1(\bar x) = \phi_2(\bar y)$. Since $n \ge l$, we have ${\bar y}_{[-l,l]} = y_{[-l,l]}$, so $(\bar x,\bar y)$ is in $C$ satisfying $\psi_2(\bar x,\bar y) = \bar y$. Thus $\psi_2(C)$ contains ${}_{-l}[y_{-n-m} \cdots y_{n+m}]$, as desired.

    (2) Suppose $\phi_1$ is not open. Then there is an open set $U$ of $X$ such that $\phi_1(U)$ is not open. Hence we can find an $x \in U$ and a sequence $\set{z^{(n)}}_{n=1}^{\infty}$ in $Z$ such that $z^{(n)} \notin \phi_1(U)$ for any $n$ and $z^{(n)} \to \phi_1(x)$. Take points $y^{(n)} \in {\phi_2^{-1}}(z^{(n)})$. By taking a subsequence, we may assume $y^{(n)}$ converges to a point $y \in Y$. Then $\phi_2(y) = \phi_1(x)$, so $(x,y)$ is in $\Sigma$. Now consider an open set $V = ( U \times Y ) \cap \Sigma$. Since $(x,y) \in V$, it follows that $\psi_2(V)$ contains $y$.
    But $y^{(n)}$ is not in $\psi_2(V)$ for any $n$, since $$(\phi_2 \circ \psi_2)(V) = (\phi_1 \circ \psi_1)(V) = \phi_1(U).$$
    Thus $\psi_2(V)$ is not open. Therefore $\psi_2$ is not an open code.
\end{proof}

\begin{lem} \label{lem:open_has_degree}
    Let $\phi$ be a \fto\ open code from a shift space $X$ to an irreducible shift space $Y$. Then there is $d > 0$ such that $\preimageAbs{\phi}{y} = d$ for each doubly transitive point $y$ of $Y$. Furthermore, $\preimageAbs{\phi}{y} \leq d$ for all $y \in Y$.
\end{lem}
\begin{proof}
    Take a doubly transitive point $z$ of $Y$ and let $d = \preimageAbs{\phi}{z}$. Suppose that there exists a point $y \in Y$ such that $\preimageAbs{\phi}{y} > d$. Take distinct points $u_1 , \cdots , u_{d+1} \in \preimage{\phi}{y}$. Since $X$ is a compact metric space, we can take disjoint neighborhoods $U_1 , \cdots , U_{d+1}$ of $u_1 , \cdots , u_{d+1}$, respectively. Take a sequence $\{y^{(i)}\}_{i=1}^{\infty} \subset \{\sigma^k (z) \}_{k \in Z}$ such that $y^{(i)} \rightarrow y$.

    Since $\preimageAbs{\phi}{y^{(i)}} = d$, for each $i$ there exists $j(i)$ such that $U_{j(i)} \cap \preimage{\phi}{y^{(i)}} = \emptyset$. By taking a subsequence of $i$ such that $j(i)$ is constant, without loss of generality we may assume that there exist a sequence $\{y^{(i)}\}_{i=1}^\infty$ and $j$ such that $y^{(i)} \rightarrow y$ and $U_{j} \cap \phi^{-1} (  y^{(i)}  ) = \emptyset$ for all $i$.

    Consider $\phi(U_j ) \subset Y$. Note that $y^{(i)} \rightarrow y \in \phi(U_j)$, but $y^{(i)}$ is not in $\phi(U_j)$  for any $i$, hence $\phi(U_j)$ is not open. Thus $\phi$ is not an open code and a contradiction comes. Hence $\preimageAbs{\phi}{y} \leq d$ for all $y$ in $Y$. If $\bar z$ is another transitive points of $Y$, then above argument shows that $ \preimageAbs{\phi}{\bar z} = \preimageAbs{\phi}{z} = d$. So each doubly transitive point of $Y$ has $d$ preimages.
\end{proof}

Let $\pXtoY$ be a factor code between shift spaces. If there is $d > 0$ such that every doubly transitive point of $Y$ has $d$ preimages under $\phi$, we call $d$ the \textit{degree} of $\phi$. If $X$ is an \rSFT\ and $\phi$ is \fto, then $\phi$ has a degree $d$ and $\preimageAbs{\phi}{y} \geq d$ for all $y$ in $Y$. Hence the above lemma implies that any \fto\ open code has a degree and is $d$-to-1 if the domain is an \rSFT.

We will use the following lemma to prove that any \bic\ factor code also has a well-defined degree.

\begin{lem} \label{lem:fs_iff_bic}
    Let $\pXtoY$ be a code between shift spaces. Then $\phi$ is \bic\ \ifff\ there is an $\epsilon > 0$ such that whenever $y \in Y$ and $x,\bar x \in \preimage{\phi}{y}$ with $x \neq \bar x$, we have $d(x,\bar x) \ge \epsilon$.
\end{lem}
\begin{proof}
    First, if $\phi$ is not \rc, then there are two left asymptotic points $x$ and $z$ in $X$ such that $x \neq z$ and $\phi(x) = \phi(z) = y \in Y$. Consider $\sigma^{-n}(y)$ for large $n$. Since $d(\sigma^{-n}(x), \sigma^{-n}(z)) \to 0$ as $n \to \infty$, we cannot find an $\epsilon >0$ satisfying the latter condition in the statement of the lemma. A similar argument applies when $\phi$ is not \lc.

    Next, suppose that $\phi$ is \bic. Since $\phi$ is \rc, there is an $N_1 > 0$ with the following property: If $x,y \in X$,  $\phi(x) = \phi(y)$ and $x_{[-N_1,0]}=y_{[-N_1,0]}$, then $x_{1} = y_{1}$ (by an easy compactness argument.) Since $\phi$ is \lc, there is an $N_2 > 0$ with similar property. Let $N = \max \{N_1,N_2 \}.$ It is easy to see that $\epsilon = 2^{-N}$ satisfies the condition in the statement.
\end{proof}

\begin{rem} \label{lem:open_fs_localhomeo}
    Let $\pXtoY$ be an open code between shift spaces. If $\phi$ is \bic, then it is a local homeomorphism. For, let $\epsilon > 0$ be given as in Lemma \ref{lem:fs_iff_bic}. Given $x \in X$, let $U$ be the ball of radius $ \epsilon /2$ centered at $x$. Then $\phi|_U$ is a homeomorphism onto its image.
\end{rem}

\begin{lem} \label{lem:bic_has_degree}
    Let $\phi$ be a \bic\ factor code from a shift space $X$ to an irreducible shift space $Y$. Then there is a $d > 0$ such that $\preimageAbs{\phi}{y} = d$ for each doubly transitive point $y$ of $Y$. Furthermore, $\preimageAbs{\phi}{y} \geq d$ for all $y \in Y$.
\end{lem}
\begin{proof}
    Let $z$ be a doubly transitive point of $Y$ and $y \in Y$. Then there is a sequence $\{n_k\}_{k=1}^{\infty}$ such that $\sigma^{n_k} (z) \to y$.
    Write $$\preimage{\phi}{z} = \{ z^{(1)}, \cdots, z^{(d)} \}.$$
    By compactness, there exist subsequence $\{n_{k_j}\} \subset \{n_k\}$ and points $x^{(1)}, \cdots, x^{(d)}$ such that
    $$\sigma^{n_{k_j}} (z^{(i)}) \to x^{(i)} \text{ for each } i = 1, \cdots, d.$$
    Note that $\phi(x^{(i)}) = y$ for all $i$. By Lemma \ref{lem:fs_iff_bic}, $\{\sigma^{n_{k_j}} (z^{(i)})\}_{i=1}^d$ is $\epsilon$-separated for each $j \in \setN$. It follows that $x^{(i)}$'s are distinct and hence $\preimageAbs{\phi}{y} \geq \preimageAbs{\phi}{z} = d$. If $y$ is also doubly transitive, then by the symmetric argument we get $\preimageAbs{\phi}{y} = \preimageAbs{\phi}{z} = d$.
\end{proof}

\begin{cor} \label{cor:open_bic_then_cto}
    Let $\pXtoY$ be an open code between shift spaces. If $\phi$ is \bic\ and $Y$ irreducible, then $\phi$ is \cto.
\end{cor}
\begin{proof}
    This follows from Lemma \ref{lem:open_has_degree} and Lemma \ref{lem:bic_has_degree}.
\end{proof}

\begin{prop}\label{thm:open_cto_bic}
    Let $\pXtoY$ be an open code between shift spaces. If $\phi$ is \cto, then it is \bic.
\end{prop}
\begin{proof}
    Let $\phi$ be \cto\ and $d$ the number of preimages. If $\phi$ is not \bic, then by Lemma \ref{lem:fs_iff_bic} for each $n \in \setN$, there is a point $y^{(n)} \in Y$ such that $\preimage{\phi}{y^{(n)}} = \set{ x^{n,1}, \cdots, x^{n,d}}$ with $x^{n,1}_{[-n,n]} = x^{n,2}_{[-n,n]}$. By choosing a subsequence, we can assume that $y^{(n)} \to y$, ${x^{n,1}} \to x^{(2)}$, and ${x^{n,i}} \to x^{(i)}$ for $i = 2, \cdots, d$.

    Since $\phi$ is $d$-to-1, there is $z \in X$ such that $\phi(z) = y$ and $z \neq x^{(i)}$ for any $i = 2, \cdots, d$. Also there is a neighborhood $U$ of $z$ such that $x^{n,i} \notin U$ for any $i=1, \cdots, d$ and for any $n \in \setN$. Note that $y^{(n)} \to y \in \phi(U)$ but $y^{(n)}$ is not in $\phi(U)$ for any $n$. Hence $\phi(U)$ is not open, which is a contradiction. So $\phi$ is \bic.
\end{proof}

\vspace{0.95cm}

\section{Extension by open codes}\label{sec:section2}

In \cite{Nas}, Nasu considered the \cto\ extensions of \rSFTs\ for the case when the extension is assumed to be of finite type, and obtained the following result.

\begin{prop} \cite{Nas} \label{lem:cto_implies_nonwandering}
    Let $X$ be a \SFT, $Y$ an \rSFT, and $\pXtoY$ a \cto\ code. Then $X$ is nonwandering, all components are maximal, and the restriction of $\phi$ to any component is \cto.
\end{prop}

We will show a similar structural result for the \fto\ open extensions of \rSofic s. First we recall some definitions.

Let $\phi$ be a \fto\ 1-block factor code from an \rSFT\ $X$ to a shift space $Y$. Given a word $w \in \B(Y)$, define
$$d(w) = \min_{1 \leq t \leq |w|} \abs{ \set { a \in \B_1(X) : \exists u \in \phi^{-1}(w) \text{ with } u_t = a } }$$
and $d = \min_{w \in \B(Y)} d(w)$. If a word $w$ satisfies $d(w) = d$, then it is called a \emph{magic word}. In this case, a coordinate $t$ where the minimum occurs is called a \emph{magic coordinate}. For the properties of magic words, see \cite{Kit} or \cite{LM}. A word $v \in \B(X)$ is \emph{intrinsically synchronizing} if whenever $uv$ and $vw$ are in $\B(X)$, we have $uvw \in \B(X)$. Every irreducible sofic shift has an intrinsically synchronizing word.

In \cite{Nas}, Nasu proved the following result for the case of \SFTs\ by using the notions of compatible and complete sets. We generalize this to sofic shifts.

\begin{lem} \label{lem:preimage_of_dtp_isin_nonwanderingset}
    Let $X$ be a \Sofic, $Y$ an \rSofic, and $\pXtoY$ a \fto\ factor code. If $y$ is a doubly transitive point in $Y$, then $\preimage{\phi}{y} \subset \Omega (X).$
\end{lem}
\begin{proof}
    First, we prove the case where $X$ is of finite type. Without loss of generality, we can assume $X$ is an edge shift and $\phi$ is 1-block. Suppose there is a point $x$ in $X \setminus \Omega(X)$ with $\phi(x) = y$. Then there are components $X_1$, $X_2$ of $X$ such that $x$ is left asymptotic to $X_1$ and right asymptotic to $X_2$. Hence there are $k, l \in \setZ$, $k < l$, such that $x_i \in \B_1(X_1)$ for all $i \leq k$ and $x_i \in \B_1(X_2)$ for all $i \geq l$. Since $y$ is doubly transitive and $\phi$ is \fto, it follows that $x$ is left transitive in $X_1$ and right transitive in $X_2$. Hence $\phi|_{X_1}$ and $\phi|_{X_2}$ are onto.

    Note that there is an intrinsically synchronizing word $w \in \B(Y)$ which is a magic word for both $\phi|_{X_1}, \phi|_{X_2}$. (Just find three words satisfying each condition and glue them using the irreducibility of $Y$.) Since $y$ is doubly transitive, there are $i_1, i_2, j_1, j_2 \in \setZ$ such that $i_1 < i_2 < k < l < j_1 < j_2$ and $y_{(i_1,i_2]}=y_{(j_1,j_2]} = w$. Let $\bar y = (y_{(i_1,j_1]})^\infty$. Then $\bar y \in Y$, since $w$ is intrinsically synchronizing. Note that $\phi(x_{(i_1,j_2]}) = \bar y_{(i_1,j_2]}$.

    Let $m$ and $n$ denote magic coordinates of $w$ for $\phi|_{X_1}$ and $\phi|_{X_2}$, respectively. Then there is a unique point $x^{(1)} \in X_1$ such that $\phi(x^{(1)}) = \bar y$ and $x^{(1)}_{i_1 + m} = x_{i_1 + m}$. Also, there is a unique point $x^{(2)} \in X_2$ such that $\phi(x^{(2)}) = \bar y$ and $x^{(2)}_{j_1 + n} = x_{j_1 + n}$. Now define $\bar x \in X$ by
    \begin{equation}
        {\bar x}_i =
            \begin{cases}
                x^{(1)}_i & \text{if $i \leq i_1 + m$}   \\
                x_i     & \text{if $i_1 + m \leq i \leq j_1 + n $} \\
                x^{(2)}_i & \text{if $i \geq j_1 + n $}   \\
            \end{cases}
    \end{equation}
    Note that $\bar x$ is well defined and $\phi(\bar x) = \bar y$. Also $\bar x$ is not periodic. Since $\bar y$ is periodic, it follows that $\phi$ is not \fto, which is a contradiction. Thus $\preimage{\phi}{y} \subset \Omega (X)$.

    Next, we prove the case where $X$ is sofic. Let $y$ be a doubly transitive point in $Y$. Take the canonical cover $\pi : \X_A \to X$. By applying the previous result to $\phi \circ \pi$, we get $x \in \Omega(\X_A)$ for all $x \in \pi^{-1} (\preimage{\phi}{y})$. Since the factor code preserves nonwandering points, it follows that $\pi(x) \in \Omega(X)$. Thus $\preimage{\phi}{y} \subset \Omega (X)$.
\end{proof}

\begin{prop} \label{lem:open_implies_nonwandering}
    Let $X$ be a \SFT, $Y$ an \rSofic, and $\pXtoY$ a \fto\ open code. Then $X$ is nonwandering, all components are maximal, and the restriction of $\phi$ to any component is open.
\end{prop}
\begin{proof}
    We can assume $X =  \X_A$ is an edge shift and $\phi$ is a 1-block code. Then we can write $\Omega(X) = \dot \bigcup ~ \X_{A_i}$, where $\X_{A_i}$'s are components of $\X_A$.

    Let $\X_{A_i}$ be a sink component. We claim that $\phi|_{\X_{A_i}}$ is onto. If not, then there is a word $v \in \B(Y) \backslash \B(\phi(\X_{A_i}))$. If there is no transition edge to $\X_{A_i}$, then $\X_{A_i}$ is open in $\X_A$. But $\phi(\X_{A_i})$ is a proper closed $\sigma$-invariant subset of $Y$ and it is not open in $Y$ since $Y$ is topologically transitive, thus we get a contradiction. So there is an edge $e \in \G_A$ whose terminal vertex lies in $\G_{A_i}$, but $e \notin \G_{A_i}$. Consider $\phi([e])$. Choose $l \geq 0$ so that $\phi([e])$ is a union of central $2l+1$ cylinders, one of which is $_{-l}[u]$.

    Since $Y$ is irreducible, there is a word $w$ with $u w v \in \B(Y)$. Take a point $y \in Y$ with $y_{[-l,l+|w|+|v|]} = u w v$. Then $y \in {}_{-l}[u]$. But if $x \in \B(\X_A)$ with $x_0 = e$, then $x_k \in \B(\X_{A_i})$ for all $k > 0$ and hence $v$ cannot occur in $\phi(x)_{[1,\infty)}$. It follows that $y \notin \phi([e])$, which is a contradiction. Hence $\phi|_{\X_{A_i}}$ is onto for each sink component. Similarly $\phi|_{\X_{A_j}}$ is onto for each source component $\X_{A_j}$.

    Suppose that $\Omega(\X_A) \neq \X_A$. Then there is a point $x \in \X_A \backslash \Omega(\X_A)$ such that $x$ is left asymptotic to some source component and right asymptotic to some sink component. Furthermore, we can assume that $x$ is left transitive in the source component, and right transitive in the sink component. Then $\phi(x)$ is a doubly transitive point in $Y$, which contradicts Lemma \ref{lem:preimage_of_dtp_isin_nonwanderingset}. Hence $\X_A$ is nonwandering. Now, since each component $\X_{A_i}$ is open and closed, the restriction $\phi|_{\X_{A_i}}$ must be open, hence onto by Lemma \ref{lem:open_then_onto}. Consequently, all components are maximal.
\end{proof}

We are now ready to prove the main part of Theorem \ref{thm:Main1}. The heart of the proof lies in the next two propositions. For a $0$-$1$ matrix $A$, denote $\hat \X_A$ by the shift space consisting of $(x_i)_{i \in \setZ}$ with $A_{x_i x_{i+1}} = 1$. There is a natural conjugacy $\pi : \X_A \to \hat \X_A$ \cite{LM}.

\begin{prop} \label{lem:rc_open_then_nonwandering}
    Let $\pStoY$ be a \rc\ open code from a shift space $\Sigma$ to an \rSFT\ $Y$. Then $\Sigma$ is a nonwandering \SFT.
\end{prop}
\begin{proof}
    By recoding, we can assume that $Y$ is an edge shift and $\phi$ is a 1-block \rr\ code. Define a 0-1 matrix $A$, indexed by symbols of $\Sigma$, by $A_{ij} = 1$ if  $ij \in \B(\Sigma)$ and $A_{ij}=0$ otherwise. We can extend $\phi$ to a \rr\ code $\bar \phi : \hat \X_A \to Y$ by letting $\bar \phi (x)_i = \phi(x_i)$ for $i \in \setZ$. Since $Y$ is an edge shift, $\bar \phi$ is well defined \cite[Theorem 4.12]{BMT}. The existence of $\bar \phi$ implies that every maximal component of $\hat \X_A$, hence of $\X_A$, is a sink component \cite[Lemma 5.1.4]{Kit}. By precomposing the conjugacy $\pi : \X_A \to \hat \X_A$, we can assume that $\bar \phi$ is a code from $\X_A$ to $Y$ which is an extension of $\phi$.

    Let $\X_{A_i}$ be a maximal component of $\X_{A}$. If $\Sigma \cap \X_{A_i} \subsetneq \X_{A_i}$, then take words $v \in \B(Y) \setminus  \B(\phi(\X_{A_i} \cap \Sigma))$ and $e \in \B_1(\X_{A_i})$.
    As in the proof of Proposition \ref{lem:open_implies_nonwandering}, $\phi([e])$ is a union of $2l+1$ cylinders, one of which is ${}_{-l}[u]$. Take a point $y \in Y$ with $y_{[-l,l+|w|+|v|]} = u w v$ for some $w \in \B(Y)$. Then $y \in {}_{-l}[u]$, but it cannot be in $\phi([e])$, hence we get a contradiction. Thus $\Sigma \cap \X_{A_i} = \X_{A_i}$ for each maximal component.

    Consider the subgraph $\G_{\bar A}$ of $\G_A$ consisting of all nonmaximal components of $\G_A$ and transition edges between these components. If $\X_{\bar A}$ is nonempty, then since every component of $\G_{\bar A}$ is not maximal, $\phi(\X_{\bar A})$ is not equal to $Y$. Take words $v \in \B(Y) \setminus  \B(\phi(\X_{\bar A}))$ and $e \in \B_1(\X_{\bar A})$. As above, $\phi([e])$ does not contain a point of the form $\cdots v \cdots . \phi(e) \cdots$, so $\phi$ cannot be open. Hence $\X_{\bar A}$ is empty and each maximal component is also a source component. Thus, $\Sigma = \dot \bigcup ~ \X_{A_i}$ is a nonwandering shift of finite type.
\end{proof}

\begin{prop} \label{lem:rc_cto_then_nonwandering}
    Let $\pStoY$ be a \rc\ \cto\ code from a shift space $\Sigma$ to an \rSFT\ $Y$. Then $\Sigma$ is a nonwandering \SFT.
\end{prop}
\begin{proof}
    By proceeding as in the proof of Proposition \ref{lem:rc_open_then_nonwandering}, we obtain a \SFT\ $\X_A$ and $\bar \phi : \X_A \to Y$ where each maximal component of $\X_A$ is a sink and $\bar \phi$ is a right resolving extension code of $\phi$.

    Write $$\X_A =  (\dot \bigcup ~ \X_{A_i}  ) ~ \dot \cup ~ \X_{\bar A} ~\dot \cup ~\setS$$
    where $\X_{\bar A}$ is described as in the proof of Proposition \ref{lem:rc_open_then_nonwandering} and $\setS$ is the set of all points which are left asymptotic to $\X_{\bar A}$ and right asymptotic to some maximal component $\X_{A_i}$. Also write
    $$\Sigma = ( \dot \bigcup ~ \Sigma_i ) ~  \dot \cup ~ \tilde \Sigma ~ \dot \cup ~ \tilde \setS$$ where $\Sigma_i = \Sigma \cap \X_{A_i} , \tilde \Sigma = \Sigma \cap \X_{\bar A}$ and $\tilde \setS = \Sigma \cap \setS$. Let $p$ be the number of maximal components in $\X_A$.

    Since $\phi$ is onto, there is at least one $i$ satisfying $\Sigma_i = \X_{A_i}$. For, if there is no such $i$, then since $\tilde \setS \subset \Sigma \setminus \Omega(\Sigma)$ and the topological entropy is concentrated on the nonwandering set, we have $$h(\Sigma) = \max \{h(\Sigma_1), \cdots, h(\Sigma_p), h(\tilde \Sigma) \} < h(\X_A),$$ which is a contradiction.
    Without loss of generality, we can assume the set of such $i$'s are $\{1, 2, \cdots, q\}$ with $q \leq p$.
    Since $h(\X_{\bar A}) < h(Y)$, $\bar \phi|_{\X_{\bar A}}$ is not onto. Hence there exists a doubly transitive point $y$ in $Y \setminus \bar \phi(\X_{\bar A})$. Since $\preimage{\bar \phi}{y} \subset \Omega(\X_A)$ by Lemma \ref{lem:preimage_of_dtp_isin_nonwanderingset}, it follows that $\bar \phi(\X_A \setminus \bigcup_{i=1}^p \X_{A_i}) \neq Y$ and hence $\phi(\Sigma \setminus \bigcup_{i=1}^p \Sigma_i) \neq Y$. Since $h(\Sigma_i) < h(Y)$ for each $i=q+1, \cdots, p$, it follows that $\preimage{\phi}{y} \cap \Sigma_i = \emptyset$. Thus we get $\phi(\Sigma \setminus \bigcup_{i=1}^q \Sigma_i) \neq X$.

    For each $i = 1, \cdots, q$, $\phi|_{\Sigma_i}$ is a \fto\ factor code between \rSFTs. Let $d_i$ be the degree for $\phi|_{\Sigma_i}$ and $d = \sum_{i=1}^q d_i$. If $\Sigma \setminus \bigcup_{i=1}^{q} \Sigma_i \neq \emptyset$, then take a doubly transitive point $z$ in $Y \setminus \phi(\bigcup_{i>q} \Sigma_i \cup \tilde \Sigma)$. By Lemma \ref{lem:preimage_of_dtp_isin_nonwanderingset}, we get $\preimage{\bar \phi}{z} \subset \Omega(\X_A)$, hence $\preimage{\phi}{z} \subset ( \bigcup_{i=1}^{p} \Sigma_i) \cup \tilde \Sigma $. By the condition on $z$, $\preimage{\phi}{z}$ lies in $\bigcup_{i=1}^{q} \Sigma_i$, so $\preimageAbs{\phi}{z} = d$ and $\phi$ is $d$-to-$1$ everywhere. But if we take $x \in \phi(\Sigma \setminus \bigcup_{i=1}^{q} \Sigma_i)$, then by letting $\phi_i = \phi|_{\Sigma_i}$ and $\tilde \phi = \phi|_{\Sigma \setminus \bigcup_{i=1}^{q} \Sigma_i}$, it follows that
    $$\preimageAbs{\phi}{x} = \sum_{i=1}^q | {\phi_i}^{-1} (x) | + | {\tilde \phi}^{-1} (x) | \geq d + 1,$$
    which is a contradiction. Hence $\Sigma \setminus \bigcup_{i=1}^{q} \Sigma_i = \emptyset$.
    Thus $$\Sigma = \bigcup_{i=1}^{q} \Sigma_i = \bigcup_{i=1}^{q} \X_{A_i}$$ so that $\Sigma$ is a nonwandering \SFT.
\end{proof}

\begin{rem}
    If $\phi$ is assumed to be \bic\ in the above propositions, then the proofs can be organized to be simpler. First, in the proof of Proposition \ref{lem:rc_open_then_nonwandering}, if we assume $\phi$ is \bir\ by recoding, then $\bar \phi$ is \bir. Hence each maximal component $\X_{A_i}$ is a sink and a source, so an open and closed set in $\X_A$. Hence $\X_A = (\dot \bigcup \X_{A_i}) \dot \cup \X_{\bar A}$. Since $\phi(\Sigma \cap \X_{A_i})$ is open, we get $\Sigma \cap \X_{A_i} = \X_{A_i}$ for each $i$. Since $\phi(\Sigma \cap \X_{\bar A})$ is open, we get $\X_{\bar A} = \emptyset$, so that $\Sigma = \bigcup \X_{A_i}$.

    Next, in the proof of Proposition \ref{lem:rc_cto_then_nonwandering}, we have a representation $\X_A = (\dot \bigcup_{i=1}^p \X_{A_i}) \dot \cup \X_{\bar A}$, where $\X_{A_i}$'s are maximal components and $h(\X_{\bar A}) < h(X)$. Let $\Sigma = (\dot \bigcup_{i=1}^p \Sigma_i) \dot \cup \tilde \Sigma$ as in the proof. Since $\phi$ is onto, $\Sigma_i = \X_{A_i}$ for some $i$. By reordering, we can assume $i=1$. Now take $\Sigma^{(1)} = (\bigcup_{i=2}^p \Sigma_i) \cup \tilde \Sigma$. If $\Sigma^{(1)}$ is empty, we are done. Otherwise, $\phi|_{\Sigma_1}$ is \cto, since it is a \bic\ code between two \rSFTs. It follows that $\phi|_{\Sigma^{(1)}}$ is also \cto, so it is onto and $\Sigma_i = \X_{A_i}$ for some $i \geq 2$, which we can assume to be $2$. Now consider $\Sigma^{(2)}= (\bigcup_{i=3}^p \Sigma_i) \cup \tilde \Sigma$, and so on. By processing, we terminate after at most $p$ steps and get $\Sigma = \bigcup_{i=1}^q \X_{A_i}$ for some $q \leq p$.
\end{rem}

\vspace{1cm}

\section{Main Theorems}\label{sec:main}

As we have seen, there are relations among open, \cto, and \bic\ codes. Some of such relations have been considered in \cite{CP,Hed,Nas}. Especially, in the category of \rSFTs, they are all equivalent. As we shall see, in the case of general shift spaces, any two of the above properties imply the third. When the range is sofic, the code forces the domain to be sofic and nonwandering. It turns out that the existence of a cross section of a code is useful.

We first give the well-known result of Nasu.

\begin{thm} \cite{Nas} \label{thm:Nasu}
    Let $X$ and $Y$ be \rSFTs\ \wee\ and $\pXtoY$ a factor code. Then \TFAE.
        \begin{enumerate}
            \item $\phi$ is open.
            \item $\phi$ is \cto.
            \item $\phi$ is \bic.
        \end{enumerate}
\end{thm}
\begin{proof}
If $\phi$ is open, then it is \cto\ by the remark following Lemma \ref{lem:open_has_degree}. Suppose $\phi$ is constantly $d$-to-$1$. Without loss of generality, we can assume that $X$ is an edge shift and $\phi$ is 1-block. It is well known that the preimage of each point in $Y$ contains a set of $d$ mutually separated points \cite{LM}. Since $\preimageAbs{\phi}{y} = d$, $\preimage{\phi}{y}$ is $1$-separated set for all $y \in Y$. So by Lemma \ref{lem:fs_iff_bic}, $\phi$ is \bic. Finally, if $\phi$ is \bic, then by recoding it is conjugate to a \bir\ factor code. By the Perron-Frobenius theory, it is a bi-covering code \cite[Theorem 8.2.2]{LM}, which is clearly open.
\end{proof}

These conditions do not coincide when we are in the sofic category. For the case when $Y$ is sofic, the following theorem holds. For the case when $X$ is irreducible, the result appears in \cite{CP} without proof.

\begin{thm} \label{thm:Y_Sofic}
    Let $\phi$ be a \fto\ code from a \SFT\ $X$ to an \rSofic\ $Y$. Then $\phi$ is open \ifff\ it is \cto.
    If $\phi$ is open, or equivalently, \cto, then $\phi$ is \bic.
\end{thm}
\begin{proof}
    First, suppose that $\phi$ is \cto\ and let $\pi : \X_B \to Y$ be the canonical cover so that $\X_B$ is irreducible. Consider the fiber product $(\Sigma,\psi_1,\psi_2)$ of $(X,\phi)$ and $(\X_B,\pi)$. Then $\psi_2$ is \cto. Since $X$ and $\X_B$ are of finite type, $\Sigma$ is also of finite type. Now by Proposition \ref{lem:cto_implies_nonwandering}, the restriction of $\psi_2$ to each component is \cto, hence is open and \bic\ by Theorem \ref{thm:Nasu}. So $\psi_2$ itself is open and \bic. Thus $\phi$ is also open and \bic\ by Lemma \ref{lem:open_fiberproduct}.

    Conversely, suppose  $\phi$ is open. By Lemma \ref{lem:open_implies_nonwandering}, $X$ is nonwandering and the restriction of $\phi$ to each component is open. By Proposition \ref{prop:open_sft_then_sft}, $Y$ is of finite type. Thus Theorem \ref{thm:Nasu} applies and $\phi$ is \cto.
\end{proof}

We prove the main theorem of this paper.

\begin{proof}[Proof of Theorem \ref{thm:Main1}]
    First, suppose $\phi$ is open and \rc. Let $\pi : \X_B \to Y$ be the canonical cover of $Y$ so that $\X_B$ is irreducible. Consider the fiber product of $(\Sigma,\psi_1,\psi_2)$ of $(X,\phi)$ and $(\X_B,\pi)$.

        \begin{diagram}
                &               &\Sigma     &               &\\
                &\ldTo^{\psi_1} &           &\rdTo^{\psi_2} &\\
            X   &               &           &               &\X_B\\
                &\rdTo_{\phi}   &           &\ldTo_{\pi}    &\\
                &               &Y
        \end{diagram}

    By Lemma \ref{lem:open_fiberproduct}, $\psi_2$ is open and \rc, and by Proposition \ref{lem:rc_open_then_nonwandering}, $\Sigma$ is indeed a \SFT. Now by Proposition \ref{lem:open_implies_nonwandering}, $\Sigma$ is nonwandering, and the restriction of $\psi_2$ to each component is open, hence \cto\ and \bic\ by Theorem \ref{thm:Nasu}. So $\psi_2$ itself is \cto\ and \bic. Hence $\phi$ is also \cto\ and \bic.

    Next, if $\phi$ is \cto\ and \rc, then the proof is similarly established, by using Propositions \ref{lem:rc_cto_then_nonwandering}, \ref{lem:cto_implies_nonwandering}, and Lemma \ref{lem:open_fiberproduct}.

    The case where $\phi$ is open and \cto\ holds by Proposition \ref{thm:open_cto_bic}.

    Finally, note that $\psi_1$ is onto, since $\pi$ is onto. Hence if all of these conditions hold, then $X$ is a factor of a nonwandering \SFT\ $\Sigma$, so it is nonwandering sofic.
\end{proof}

A shift space $X$ is called \textit{almost Markov} if it is a \bic\ factor of a \SFT. We call $X$ a \emph{strictly almost Markov shift} if it is almost Markov but not of finite type. An \textit{almost of finite type shift} (AFT) is an irreducible almost Markov shift \cite{BoyK}. It is known that an \rSofic\ is an AFT \ifff\ its canonical cover is also \lc\ \cite{BKM}.

\begin{cor}\label{cor:rcopen_ext}
    The following statements hold.
    \begin{enumerate}
        \item
            A \rc\ open extension of an \rSFT\ is of finite type.
        \item
            A \rc\ open extension of an irreducible strictly almost Markov shift is strictly almost Markov.
        \item
            A \rc\ open extension of an irreducible non-AFT \Sofic\ is sofic and not almost Markov.
    \end{enumerate}
    Each extension is nonwandering and all irreducible components in the extensions are maximal. The same results hold for \rc\ \cto\ extensions.
\end{cor}
\begin{proof}
    (1) is merely a restatement of Proposition \ref{lem:rc_open_then_nonwandering}.

    (2) If $\pXtoY$ is a \rc\ open code from a shift space to an AFT, then $\pi$ in the proof of Theorem \ref{thm:Main1} is \bic. Hence $\psi_1$ is \bic, so that $X=\psi_1(\Sigma)$ is an almost Markov shift. Next, if $X$ is of finite type, then $Y$ is also of finite type by Proposition \ref{prop:open_sft_then_sft}, which is a contradiction. So $X$ is strictly sofic.

    (3) Let $\pXtoY$ is a \rc\ open code from a shift space to an irreducible non-AFT \Sofic. Note that the extension shift is sofic and $\phi$ is \bic\ by Theorem \ref{thm:Main1}. If $X$ is almost Markov, then $Y$ has a \bic\ cover, namely, the \bic\ cover of $X$ followed by $\phi$, which is a contradiction. Hence $X$ is not almost Markov.

    The remainder of the statement follows from Theorem \ref{thm:Main1}.
\end{proof}

As stated in the introduction, the above result can be viewed as a generalization of  \cite{BH}. In a sense, \rc\ open (or \cto) extension preserves the structure of a shift space.

\begin{rem}
    The code constructed in \cite{Fie} is a \cto\ code from a nonsofic coded system to an \rSFT. By Corollary \ref{cor:rcopen_ext} it is not closing. Applying Theorem \ref{thm:Main1} it is not open.
\end{rem}

Now we consider codes between general shift spaces. If we weaken the condition (3) in Theorem \ref{thm:Main1}, then the following generalization holds. In contrast to the sofic case, we cannot use the theorem of Nasu, since there is no way to reduce to the finite type situation.

\begin{thm} \label{thm:Main2}
    Let $\phi$ be a code from a shift space $X$ to an irreducible shift space $Y$. Then any two of the following conditions imply the third:
        \begin{enumerate}
            \item $\phi$ is open.
            \item $\phi$ is \cto.
            \item $\phi$ is \bic.
        \end{enumerate}
\end{thm}

\begin{proof}
    The theorem follows from Corollary \ref{cor:open_bic_then_cto}, Proposition \ref{thm:open_cto_bic}, and the following proposition \ref{thm:cto_then_open_iff_bic}.
\end{proof}

Let $\phi : X \to Y$ be a continuous map between topological spaces. A continuous map $f : Y \to X$ is called a \emph{cross section} of $\phi$ if $\phi(f(y)) = y$ for all $y$ in $Y$. Any open factor code between shift spaces has a cross section \cite{Hed}. We say $\phi$ has $d$ \emph{disjoint cross sections} if there exist $d$ cross sections $f_i : Y \to X$ such that $f_i (Y) \cap f_j (Y) = \emptyset$ for all $i \neq j$. We adopt the ideas from \cite{Kur} to prove the following result.

\begin{prop} \label{thm:cto_then_open_iff_bic}
    Let $\pXtoY$ be a $d$-to-$1$ code between shift spaces. Then the following are equivalent.
    \begin{enumerate}
        \item \label{cross-opencto}
            $\phi$ is open.
        \item \label{cross-fscto}
            $\phi$ is \bic.
        \item \label{cross-dcross}
            $\phi$ has $d$ disjoint cross sections such that the union of their images is $X$.
        \item \label{cross-ontocross}
            For any $x \in X$, there exists a cross section $f$ such that $x \in f(Y)$.
    \end{enumerate}
\end{prop}
\begin{proof}
    (\ref{cross-opencto}) $\To$ (\ref{cross-fscto}): This holds by Proposition \ref{thm:open_cto_bic}.

    (\ref{cross-fscto}) $\To$ (\ref{cross-dcross}): We can assume that $\phi$ is a 1-block code. Since $\phi$ is \bic, we can find an $\epsilon > 0$ satisfying the condition in Lemma \ref{lem:fs_iff_bic}. Take $p \in \setN$ with $2^{-p} < \epsilon$.

    For each $n \geq p$ and $w \in \B_{2n+1}(Y)$, define
    $$\D(w) = \set{x_{[-p,p]} : x \in X,\phi(x_{[-n,n]}) = w} \subset \B_{2p+1}(X)$$
    and
    $$d(w) = |\D(w)|.$$
    Note that $d(w) \geq d$ for all $w$ by Lemma \ref{lem:fs_iff_bic}. Now define
    $$Y_n = \set{ y \in Y : d(y_{[-n,n]}) > d }.$$
    Each $Y_n$ is a closed set and the family of $Y_n$'s is nested. If $Y_n \neq \emptyset$ for all $n$, then any $y \in \bigcap Y_n$ has at least $d+1$ preimages by compactness and Lemma \ref{lem:fs_iff_bic}, which is a contradiction. So there is an $n \geq p$ with $Y_n = \emptyset$. Hence $d(w) = d$ for all $w \in \B_{2n+1}(Y)$.

    By the choice of $n$, we can define $g_1, \cdots, g_d : \B_{2n+1}(Y) \to \B_{2p+1}(X)$ satisfying
    $$\set{g_1(w), \cdots, g_d(w)} = \D(w).$$
    Finally, we define cross sections $f_1, \cdots, f_d : Y \to X$ as follows: For each $i = 1, \cdots, d$, $f_i (y) \text{ is the unique point in } \preimage{\phi}{y} \text{ such that } f_i (y)_{[-p,p]} = g_i (y_{[-n,n]}).$ Then all conditions are clearly satisfied except continuity.

    Fix $i$ and define an open and closed set
    $$V_i = \bigcup \set{ {}_{-n}[u] : u \in \B_{2n+1}(X), (g_i \circ \phi)(u) = u_{[n+1-p,n+1+p]} }.$$
    Then $x \in V_i$ \ifff\ $x = (f_i \circ \phi) (x)$. So $f_i^{-1} = \phi|_{V_i} : V_i \to Y$ is a bijective continuous map between compact metric spaces. Hence $f_i$ is continuous, as desired.

    It is clear that (\ref{cross-dcross}) implies (\ref{cross-ontocross}).

    (\ref{cross-ontocross}) $\To$ (\ref{cross-opencto}): Suppose that $\phi$ is not open.
    Then there is an open set $U$ of $X$ such that $\phi(U)$ is not open. Hence we can find a point $x \in U$ and a sequence $\set{y^{(n)}}_{n=1}^{\infty}$ in $Y$ such that $y^{(n)} \notin \phi(U)$ for any $n \in \setN$ and $y^{(n)} \to \phi(x)$. Find a cross section $f$ with $x \in f(Y)$. Then $x = (f \circ \phi) (x)$. Let $z^{(n)} = f (y^{(n)})$ for each $n$. Then by the continuity of $f$, we have $z^{(n)} \to x$. But since $U$ is open, $z^{(n)} \in U$, and hence $y^{(n)} \in \phi(U)$  for all large $n$, which is a contradiction. So $\phi$ is open.
\end{proof}

\begin{rem}
    In Theorem \ref{thm:Main2}, the irreducibility of $Y$ is needed only to prove that when $\phi$ is open and \bic, it is \cto. If $Y$ is not irreducible, then there is a simple counterexample: $X = \{ (01)^\infty, (10)^\infty, 2^\infty\}$, $Y=\{ a^\infty, b^\infty \}$, and $\phi((01)^\infty) = \phi((10)^\infty) = a^\infty, \phi(2^\infty) = b^\infty$.
\end{rem}

\vspace{1cm}

\section{Examples} \label{sec:examples}

In sofic category, all the three conditions in Theorem \ref{thm:Nasu} are different, that is, there is no implication between any two of those conditions. We give three examples, in each of which either $X$ or $Y$ (or both) is strictly sofic.

\begin{exam} \label{ex:only_bic}
    Let $Y$ be the even shift with its canonical cover $\phi : \X_A \to Y$, where $\X_A$ is the edge shift formed by the underlying graph in Figure \ref{fig:only_bic} and $\phi$ is given by the labeling. Note that $\phi$ is \bir\ and so $Y$ is an AFT.

    \begin{Figure}[h]
        \includegraphics[height=2cm]{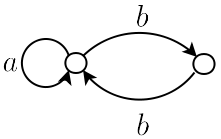}
        \label{fig:only_bic}
        \caption{A presentation of $Y$}
    \end{Figure}

    Since $\preimageAbs{\phi}{y} = 1$ for all $y \in Y$ except $y = b^\infty$, it follows that $\phi$ is not \cto. Also, by Proposition \ref{prop:open_sft_then_sft}, $\phi$ is not open. Thus $\phi$ is a \bic\ code which is neither open nor \cto.

\end{exam}

\begin{exam} \label{ex:only_open}
    Let $X$ be the sofic shift obtained by identifying two fixed points in the full 2-shift, $Y = \{a,b\}^\setZ$ the full 2-shift, and $\pXtoY$ the subscript dropping code (see Figure \ref{fig:only_open}).

    \begin{Figure}[h]
        \includegraphics[height=2cm]{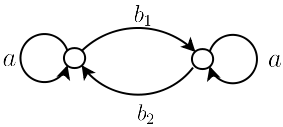}
        \label{fig:only_open}
        \caption{A presentation of $X$.}
    \end{Figure}

    Note that $\phi$ is a 1-block code and $X$ is strictly sofic. Given $u \in \B(X)$ and $l \geq 0$, let $ v = \phi(u)$. Then for given $y \in {}_l[v]$, it is easy to find an $x$ in ${}_l[u]$ such that $\phi(x) = y$. Hence $\phi({}_l[u]) = {}_l[v]$ so that $\phi$ is open.
    Every point except $a^{\infty}$ has two preimages, while $a^{\infty}$ has only one preimage. Hence $\phi$ is not \cto. Also, it identifies two points $a^{\infty}.(b_1 b_2)^{\infty}$ and $a^{\infty}.(b_2 b_1)^{\infty}$, so $\phi$ is not \rc. Similarly it is not \lc. Thus $\phi$ is an open code which is neither \cto\ nor closing.
    This example also shows that an open extension of an \rSFT\ need not be of finite type (see Corollary \ref{cor:rcopen_ext} (1)).

    We define a cross section $f_1$ as follows: Let $y \in Y$. If $y_i = b$ with $i \geq 0$, replace it with $b_1$ if the number of occurrence of $b$ in $y_{[0,i]}$ is odd and otherwise with $b_2$. If $y_i = b$ with $i < 0$, replace it to $b_2$ if the number of occurrence of $b$ in $y_{[i,0)}$ is odd and otherwise with $b_1$. The resulting point is $f_1(y)$. Define $f_2$ similarly by exchanging the roles of $b_1$ and $b_2$. Then these two cross sections cover $X$ and are disjoint except one point $a^\infty$. Hence this example satisfies the condition (4) in Proposition \ref{thm:cto_then_open_iff_bic}. Note that the implication (4) $\To$ (1) in that proposition holds for any codes (not necessarily \cto).
\end{exam}

\begin{exam} \label{ex:only_cto}
    Let $X$ be a \Sofic\ defined by the labeling in Figure \ref{fig:only_cto}, $Y$ the even shift as shown in Figure \ref{fig:only_bic}, and $\pXtoY$ the subscript dropping code. This example appears in \cite{BH}.

    \begin{Figure}[h]
        \includegraphics[height=4cm]{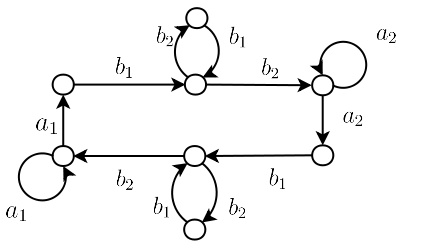}
        \label{fig:only_cto}
        \caption{A presentation of $X$.}
    \end{Figure}

    It is easy to see that $\phi$ is 2-to-1 everywhere. We will show that $\phi$ is neither closing nor open.
    First, since $\phi$ identifies the points $(b_1 b_2)^{\infty}.a_i(b_1 b_2)^{\infty}$ for $i = 1,2$, it is not closing.
    Now, if $\phi$ is open, then $\phi([b_1]) = \bigcup_{i=1}^k C_i$ for cylinders $C_i$ defined by central $2l+1$ blocks in $Y$, for some $l \geq 0$. Since $b^{\infty}$ is in $\phi([b_1])$, there exists $i$ with $b^{\infty} \in C_i$, so $C_i$ is of the form ${}_{-l}[b \cdots b]$. Now, let $y = b^{\infty}.b^{2l+1}a^{\infty}$. Note that $y$ is in $C_i$. But if $x \in X$ with $x_0 = b_1$, then the first occurrence of $a_1$ or $a_2$ in $x_{[0,\infty)}$ must be in an even coordinate. So $y$ cannot be in $\phi([b_1])$, which is a contradiction. Thus $\phi$ is not open.

    In contrast to the finite type case, where any \cto\ code has a cross section, this example shows that a \cto\ code may have no cross section in general. Indeed, if $f : Y \to X$ is a cross section, we can assume that $f(b^\infty) = (b_1 b_2)^\infty$. Since $f$ must send $a^\infty b^{2n+1} . b^\infty$ to $a_i^\infty (b_1 b_2)^n b_1 . (b_2 b_1)^\infty$, it cannot be continuous, which is a contradiction.
\end{exam}

\textit{Acknowledgment.} I am grateful to my advisor, Sujin Shin, for encouragement and guidance. Her valuable comments have improved an earlier version of this paper greatly.


\vspace{1cm}
\bibliographystyle{amsplain}

\end{document}